\documentclass{amsart}
\usepackage{latexsym}
\usepackage{amsfonts}
\usepackage{amsthm}
\usepackage{amsmath}
\usepackage{amssymb}
\usepackage{enumerate}
\usepackage{graphicx}
\newtheorem{theorem}{Theorem}[section]

\newtheorem{lemma}[theorem]{Lemma}
\theoremstyle{remark}

\newtheorem{definition}[theorem]{Definition}

\begin{document}
\title{A Degree-Theoretic Proof of a Coarse Fixed Point Principle}
\author{Steven Hair}
\date{\today}
\maketitle
\begin{abstract}
We introduce a large scale analogue of the classical fixed-point property for continuous maps, which shall apply to coarse maps. We also develop a coarse version of degree for coarse maps on Euclidean spaces. Then, applying a coarse degree-theoretic argument, we prove that every coarse map from a Euclidean half-space to itself has the coarse fixed-point property.
\end{abstract}
\section{Introduction}
This paper is intended to establish a coarse geometry analogue to the classical Brouwer fixed point theorem, which states that every continuous function from a disk to itself has a fixed point. To this end, we shall first review some of the fundamental principles of coarse geometry -- namely coarse maps, coarse homology, and coarse homotopy -- then discuss various notions of coarse fixed point properties for maps on open cones of spaces. Finally, we shall show that every coarse map from a Euclidean half-space to itself possesses such a coarse fixed point property, using a degree-theoretic argument.
\section{Coarse Maps, Degree, Homotopy}
The following definitions apply to coarse spaces equipped with a \emph{metric coarse structure}; they are, however, generalizable to the general category of coarse spaces (as in the sense of \cite{Roe1}, Section 2.1).
\begin{definition}
Let $X$ and $Y$ be metric spaces. A \emph{coarse map} $f: X \to Y$ is a function such that:
\begin{itemize}
\item If $B \subset Y$ is bounded, $f^{-1}(B) \subset X$ is bounded ($f$ is \emph{proper});
\item For every $R > 0$ there exists $S > 0$ such that, f $x,x' \in X$ with $d_X(x,x') \leq R$, then $d_Y(f(x),f(x')) \leq R$ ($f$ is \emph{bornologous}).
\end{itemize}
\end{definition}

\begin{definition}
If $X$ is a metric space, the group of \emph{coarse $q$-chains} on $X$ with coefficients in $\mathbb{Z}$, denoted $CX_q(X;\mathbb{Z})$, is the set of all locally finite linear combinations of $(q+1)$-tuples
\[
\sigma = \sum_{x=(x_0,\ldots,x_q) \in X^{q+1}} r_x \cdot (x_0,\ldots,x_q), \ r_x \in \mathbb{Z},
\]
such that there exists $R > 0$ such that, for each $(x_0,\ldots,x_q)$ in the sum, $d_X(x_i,x_j) \leq R$.

The coarse chain complex $CX_*(X;\mathbb{Z})$ is given the boundary map $b$ defined by
\[
b\left(\sum_x r_x \cdot (x_0,\ldots,x_q)\right) := \sum_{i=0}^q (-1)^i \sum_x r_x \cdot (x_0,\ldots,\hat{x_i},\ldots,x_q),
\]
and the resulting homology groups, the \emph{coarse homology groups} of $X$ with coefficients in $\mathbb{Z}$, are denoted $HX_*(X;\mathbb{Z})$.
\end{definition}
The key example in this paper is the coarse homology of the Euclidean space $\mathbb{R}^n$:
\[
HX_{q}(\mathbb{R}^n) \cong
\begin{cases}
\mathbb{Z} & \text{if }q=n,\\
0 & \text{otherwise};
\end{cases}
\]
see Chapter 2 in \cite{Roe3}. If $N$ is a neighborhood of the diagonal in $\mathbb{R}^n$ of radius at least 1, $HX_{n}(\mathbb{R}^n)$ is generated by the class of the cycle
\[
\sum_{x_i \in \mathbb{Z}^n \cap N} (x_0, \ldots, x_n).
\]
Every coarse map $f: \mathbb{R}^n \to \mathbb{R}^n$ induces a group homomorphism $f_*: HX_n(\mathbb{R}^n;\mathbb{Z}) \to HX_n(\mathbb{R}^n;\mathbb{Z})$ in the obvious way. As $f_*$ is a homomorphism from $\mathbb{Z}$ to itself, it is of the form $f_*(z) = d \cdot z$ for some $d \in \mathbb{Z}$.
\begin{definition}
The \emph{degree} of the coarse map $f$ is defined to be $d$.
\end{definition}

The following definition is due Higson, Roe, and others, and refined by Luu in \cite{Luu}:
\begin{definition}\label{homotopy}
Let $X$ and $Y$ be metric spaces. A \emph{coarse homotopy} from $X$ to $Y$ is a map $h_t: X \times [0,1] \to Y$ such that:
\begin{itemize}
\item For all bounded $B \subset Y$, the set
\[
\bigcup_{t \in [0,1]} h^{-1}_t(B)
\]
is bounded (the family $\{h_t\}$ is \emph{uniformly proper});
\item For all $R > 0$ there exists $S > 0$ such that, if $x,x' \in X$ with $d_X(x,x') \leq R$, then
\[
d_Y(h_t(x),h_t(x')) \leq S
\]
for all $t \in [0,1]$ (the family $\{h_t\}$ is \emph{uniformly bornologous});
\item There exists $R > 0$ such that for all $x \in X$, $t \in [0,1]$, there is an open neighborhood $I_{x,t} \subset [0,1]$ of $t$ such that $t' \in I_{x,t}$ implies $d(h_t(x), h_{t'}(x)) \leq R$ (the family $\{h_t\}$ is \emph{uniformly pseudocontinuous}).
\end{itemize}
\end{definition}
Two coarse maps which are coarsely homotopic induce the same homomorphism on coarse homology (cf. \cite{HigRoe}).

\section{Coarse Fixed Point Properties}
We now introduce a number of fixed point properties, definable on the open cone of a space:
\begin{definition}
Let $X$ be a compact metrizable space. The \emph{open cone on $X$}, denoted $\mathcal{O}X$, is the space obtained from $[0,1) \times X$ by collapsing $\{0\} \times X$ to a point. The open cone is given a metric coarse structure as follows. We may suppose that $X$ is a compact subset of the unit sphere in a real Hilbert space $E$; let $\phi$ be a homeomorphism of $[0,1)$ onto $[0,\infty)$. Then, the map
\[
\mathcal{O}X \to E, \ \ \ \ (t,x) \mapsto \phi(t)x,
\]
identifies $\mathcal{O}X$ with a subset of $E$ and so defines a metric on $\mathcal{O}X$.
\end{definition}
Note that $\mathbb{R}^n$ is the open cone on the $(n-1)$-sphere $S^{n-1}$, and the half-space $\mathbb{R}^n \times [0,\infty)$ is the open cone on the $n$-disk $D^n$.

\begin{definition}
A coarse map $f: \mathcal{O}X \to \mathcal{O}X$ has the:
\begin{enumerate}
\item \emph{Strong coarse fixed point property} if there exists a point $\zeta \in \mathcal{O}X$ (called the \emph{coarse fixed point} of $f$), a constant $R > 0$, and a sequence of points $x_i \to \infty$ in $\mathcal{O}X$ such that, for all $i$, both $x_i$ and $f_i$ are within distance $R$ of the ray $\mathcal{O}\zeta$;

\item \emph{Coarse fixed point property} if there exists a sequence of points $\zeta_i \in X$, a constant $R > 0$, and points $x_i \to \infty$ in $\mathcal{O}X$ such that, for all $i$, both $x_i$ and $f(x_i)$ are within distance $R$ of the ray $\mathcal{O}\zeta_i$;

\item \emph{Weak coarse fixed point property} if there exists a $\zeta \in X$ such that, for any neighborhood $U$ of $\zeta$, there is a sequence of points $x_i \to \infty$ in $\mathcal{O}X$ such that, for all $i$, $x_i$ and $f(x_i)$ lie in the cone $\mathcal{O}U$.
\end{enumerate}
\end{definition}
It can be shown (cf. \cite{Hair}, Section 3.4.1) that the Strong CFPP is stronger than the CFPP, which is in turn stronger than the Weak CFPP. For the purposes of this paper, we shall concern ourselves with the ``middle" coarse fixed point property.

\section{The Coarse Fixed Point Theorem}
\begin{theorem}
Every coarse map $f: \mathbb{R}^n \times [0,\infty) \to \mathbb{R}^n \times [0,\infty)$ has the coarse fixed point property.
\end{theorem}
\begin{proof}
The proof of this theorem is analogous to the degree-theoretic proof of the classical Brouwer fixed point theorem. Define a coarse map $g: \mathbb{R}^{n+1} \to \mathbb{R}^{n+1}$ by $g(x_1, \ldots, x_n, t) := f(x_1, \ldots, x_n, \vert t \vert)$. Since $g$ factors through a map to $\mathbb{R}^n \times [0,\infty)$, and since the coarse homology of $\mathbb{R}^n \times [0,\infty)$ vanishes in all dimensions (by an argument analogous to Proposition 10.1 in \cite{HPR}), $g$ has degree zero. We shall now prove, via a series of lemmas, that a degree zero map must have the CFPP.
\begin{lemma}
For $0 \leq i \leq n$, the $i^{\text{th}}$ reflection map from $\mathbb{R}^n$ to itself, which sends each $(x_0, \ldots, x_n)$ to $(x_0, \ldots, x_{i-1}, -x_i, x_{i+1}, \ldots, x_n)$, has degree $-1$.
\end{lemma}
\begin{proof}
It suffices to prove the result for $r_1$, as the general proof is analogous. Take a locally finite triangulation $\Delta$ of $\mathbb{R}^n$ such that the interior of each $n$-simplex is disjoint from $X$. We may express $\mathbb{R}^n$ as a union of spaces $D_1 = \{(x_1,x_2,\ldots,x_n) \vert x_1 \leq 0\}$, $D_2 = \{(x_1,x_2,\ldots,x_n) \vert x_1 \geq 0\}$. Letting
\[
\Delta_1 := \sum_{\sigma \in \Delta \cap D_1} \sigma, \ \Delta_2 := \sum_{\sigma \in \Delta \cap D_2} \sigma,
\]
the chain $z = \Delta_1 - \Delta_2$ is a cycle representing the generator of $HX_n(\mathbb{R}^n)$ (identifying each $n$-chain $[a_0,\ldots,a_n]$ with the point $(a_0,\ldots,a_n)$). Since $r$ induces the chain map sending $\Delta_1$ to $\Delta_2$ and vice versa, this induced map sends $z$ to $-z$. Thus $r$ has degree $-1$.
\end{proof}
\begin{lemma}
The antipodal map from $\mathbb{R}^{n+1}$ to itself, which sends each $x = (x_0, \ldots, x_n)$ to $-x = (-x_0, \ldots, -x_n)$, has degree $(-1)^{n+1}$.
\end{lemma}
\begin{proof}
The antipodal map is the composition of $n+1$ reflection maps.
\end{proof}
\begin{lemma}
Suppose a coarse map $h: \mathbb{R}^{n+1} \to \mathbb{R}^{n+1}$ does not have the CFPP. Then the linear homotopy from each $-x \in \mathbb{R}^{n+1}$ to $h(x)$ defines a coarse homotopy from the antipodal map to $h$, and so $h$ has degree $(-1)^{n+1}$.
\end{lemma}
\begin{proof}
Define a homotopy $H_t: \mathbb{R}^{n+1} \times [0,1] \to \mathbb{R}^{n+1}$ by $H_t(x) := t \cdot h(x) - (1-t)x$. We need to show that $H_t$ satisfies the three requirements for a coarse homotopy. First, since $H_t$ is a continuous homotopy, $\{H_t\}$ is clearly uniformly pseudocontinuous. Also, if $d(x,x') \leq R$, then $d(-x,-x') \leq R$ and there exists $S$ such that $d(h(x),h(x')) \leq S$ (since $h$ is coarse). Therefore, for all $t \in [0,1]$,
\[
d(H_t(x),H_t(x')) \leq t\cdot d(h(x),h(x')) + (1-t)d(x,x') \leq R + S;
\]
thus $\{H_t\}$ is uniformly bornologous. Finally, to show that $\{H_t\}$ is uniformly proper, let $B \subset \mathbb{R}^{n+1}$ be bounded; by replacing $B$ with a larger set if necessary, we may assume that $B$ is a ball of radius $T$ centered at the origin. To prove uniform properness it suffices to show that
\[
V_B := \bigcup_{t \in [0,1]} H^{-1}_t(B)
\]
is bounded. Observe first that $x \in V_B$ if and only if the line segment $L_x$ connecting $-x$ and $f(x)$ intersects $B$. We claim that there exists a $C > 0$ such that, for any $x$ such that $L_x$ intersects $B$, $h(x)$ is within distance $CT$ of the ray originating at the origin and containing $x$. Since we have assumed that $h$ does not have the CFPP, this would prove that the set $\{x \ \vert \ L_x \cap B \neq \emptyset\}$ is bounded.

Suppose that $d(x,0) \geq 2T$. Since $f$ is coarse, there exists $A > 0$, independent of $x$, such that
\[
d(h(x),0) \leq d(h(x),h(0)) + d(h(0),0) \leq A \cdot d(x,0) + d(f(0),0).
\]
Therefore, there exists a constant $K$ such that $d(h(x),0) \leq K \cdot \max\{d(x,0),1\}$; let $C = 2(1+K)$. If there exists a point $p \in B \cap L_x$, then we may construct two similar triangles. The first, smaller, triangle has vertices $-x$, $0$, and $p$. By appropriately choosing a point $q$ on the ray originating at 0 and containing $x$, we construct the second, larger, triangle with vertices $-x$, $h(x)$, and $q$, so that the leg from $h(x)$ to $q$ is parallel to the leg from $p$ to $0$ (cf. Figure \ref{fig1}). Note that, by the triangle inequality, $d(-x,p) \geq d(x,0)/2$, while $d(-x,h(x)) \leq (K+1)d(x,0)$ by definition of $K$. Therefore, by similarity,
\[
d(h(x),q) = \dfrac{d(-x,h(x))}{d(-x,p)} \cdot d(p,0) \leq 2(K+1)T \leq CT.
\]
Therefore, $\{H_t\}$ is indeed uniformly proper.
\begin{figure}
\center{\includegraphics[scale=0.3]{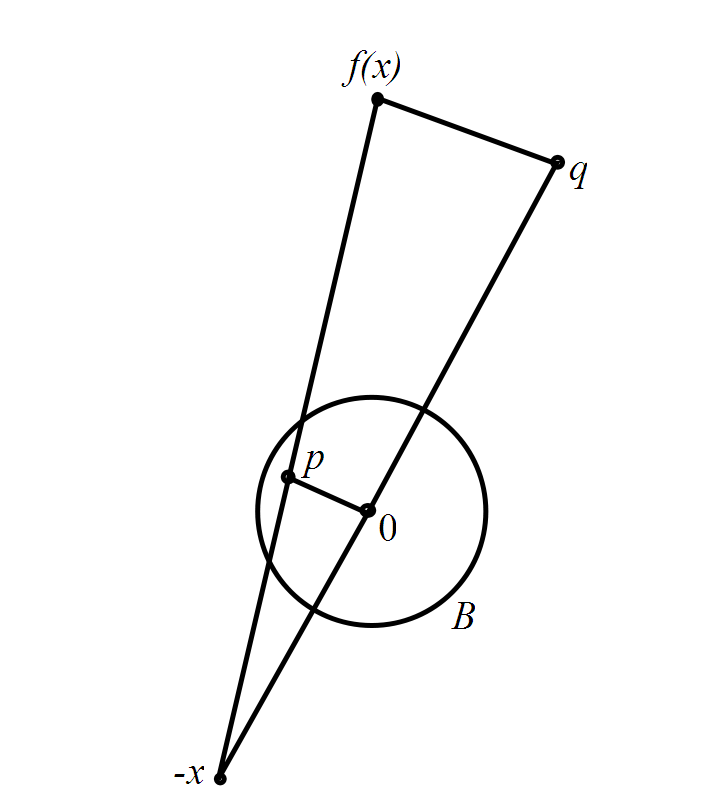}}
\caption{\label{fig1}Construction of similar triangles.}
\end{figure}
\end{proof}
Since $g$ has degree zero, it follows from the previous lemma that it has the CFPP; by the construction of $g$, we conclude that $f$ has the coarse fixed point property as well.
\end{proof}
\newpage
\bibliographystyle{amsplain}
\bibliography{coarsefixedpoint}
\end{document}